\documentclass{amsart}
\usepackage{booktabs}
\usepackage{amssymb}
\makeindex

\usepackage{pb-diagram,pb-xy} 

\usepackage{hyperref}
\hypersetup{
    colorlinks,%
    citecolor=blue,%
    filecolor=blue,%
    linkcolor=blue,%
    urlcolor=blue
}

\usepackage[backrefs]{amsrefs}

\usepackage[all]{xy}
\usepackage{tikz}
\usepackage{todonotes}
\usepackage{enumerate}
\usepackage{verbatim}
\usepackage{multirow}
\usepackage{pgf,tikz}
\usetikzlibrary{arrows}

\CompileMatrices
\newtheorem{introthm}{Theorem}

\newtheorem{theorem}{Theorem}[section]
\newtheorem{lemma}[theorem]{Lemma}

\newtheorem{proposition}[theorem]{Proposition}
\newtheorem{corollary}[theorem]{Corollary}
\theoremstyle{definition}
\newtheorem{definition}[theorem]{Definition}
\newtheorem{example}[theorem]{Example}

\newtheorem{remark}[theorem]{Remark}

\def\cc{{\mathbb C}}
\def\kk{{\mathbb K}}
\def\zz{{\mathbb Z}}
\def\rr{{\mathbb R}}

\def\qq{{\mathbb Q}}
\def\pp{{\mathbb P}}

\def\Osh{{\mathcal O}}

\def\Aut{\operatorname{Aut}} 
 
\def\Pic{\operatorname{Pic}} 
\def\Nef{\operatorname{Nef}} 
 
\def\Alb{\operatorname{Alb}} 
\def\alb{\alpha}

\begin{document}

\title{About the semiample cone of the symmetric product of a curve}

\author{Michela Artebani}
\address{
Departamento de Matem\'atica, \newline
Universidad de Concepci\'on, \newline
Casilla 160-C,
Concepci\'on, Chile}
\email{martebani@udec.cl}

\author{Antonio Laface}
\address{
Departamento de Matem\'atica, \newline
Universidad de Concepci\'on, \newline
Casilla 160-C,
Concepci\'on, Chile}
\email{alaface@udec.cl}

\author{Gian Pietro Pirola}
\address{
Dipartimento di Matematica, \newline
Universit\`a di Pavia, \newline
via Ferrata 1,
27100 Pavia, Italia}
\email{gianpietro.pirola@unipv.it}

\thanks{
The first author was partially supported 
by Proyecto FONDECYT Regular N. 1130572.
The second author was partially supported 
by Proyecto FONDECYT Regular N. 1150732.
The third author is partially supported by Indam GNSAGA 
and by Prin 2012 (Project ``Moduli di strutture geometriche e loro applicazioni'').
}

\numberwithin{equation}{section}

\begin{abstract}
Let $C$ be a smooth  curve
which is complete intersection of a
quadric and a degree $k>2$ surface 
in $\pp^3$ and let $C^{(2)}$ be its second
symmetric power. In this paper we study 
the finite generation of the extended 
canonical ring $R(\Delta,K) :=
\bigoplus_{(a,b)\in\zz^2}H^0(C^{(2)},a\Delta+bK)$,
where $\Delta$ is the image of the
diagonal and $K$ is the canonical divisor.
We first show that $R(\Delta,K)$ is finitely
generated if and only if the difference of
the two $g_k^1$ on $C$ is torsion non-trivial and
then show that this holds on an analytically 
dense locus of the moduli space of such curves.
\end{abstract}
\maketitle

\section*{Introduction}
Let $C$ be a smooth complex curve of genus $g>1$,
denote by $C^2$ the cartesian product of
$C$ with itself, by $C^{(2)}$ the 
second symmetric product of $C$  
and by $\pi\colon C^2\to C^{(2)}$ the
double cover $(p,q)\mapsto p+q$.
Let $\Delta$ be the image
of the diagonal via $\pi$ and  $K$ be
a canonical divisor of $C^{(2)}$.
In this paper we are interested
in the finite generation of the {\em extended
canonical ring}:
\[
 R(\Delta,K)
 \, :=\,
 \bigoplus_{(a,b)\in\zz^2}H^0(C^{(2)},a\Delta+bK).
\]
It is not difficult to show that finite generation
is equivalent to ask that the two-dimensional 
cone $\Nef(\Delta,K)$ consisting of classes 
of nef divisors within the rational vector space 
spanned by the classes of $\Delta$ and $K$ 
must be generated by two semiample classes
(see proof of Theorem~\ref{teo-1}).
The Kouvidakis conjecture~\cite{CK} states that,
for the very general curve, if $\Nef(\Delta,K)$ is 
closed then $g$ must be a square, 
i.e. $g = (k-1)^2$ for $k>1$ integer. 
We thus focus our attention to
the case of square genus and assume $C$
to be a complete intersection of a 
quadric $Q$ with a degree $k>2$ hypersurface.
We denote by $\eta_C$ the class of the
difference of the two $g_k^1$ of $C$
induced by the two rulings of $Q$, which
is trivial if $Q$ is a cone.
Our first result is the following.
\begin{introthm}
\label{teo-1}
Let $C\subseteq\pp^3$ be a smooth complete 
intersection of a quadric with a degree $k>2$
surface. Then the following are equivalent.
\begin{enumerate}
\item
$R(\Delta,K)$ is finitely generated.
\item
$\eta_C$ is torsion non-trivial.
\end{enumerate}
Moreover if both the above conditions are satisfied
then $\eta_C$ has order at least $k$.
\end{introthm}
Let $\mathcal F_k$ be the  open subset 
of the Hilbert scheme of curves of bi-degree $(k,k)$ 
of $\pp^1\times\pp^1$ consisting of smooth curves
and let $\mathcal F_k^{\rm tor}\subseteq\mathcal F_k$
be the subset consisting of curves $C$ such that
the class $\eta_C$ of the difference between the
two $g_k^1$ is torsion.
Our second theorem is the following.
\begin{introthm}
\label{teo-2}
The locus $\mathcal F_k^{\rm tor}$ is 
a countable union of subvarieties of complex dimension 
$\geq 4k-1$ and the set of subvarieties of
dimension $4k-1$ is dense in $\mathcal F_k$
in the analytic topology.
\end{introthm}
The paper is organized as follows.
In Section~\ref{sym}, after recalling some basic facts
about the symmetric product of a curve,
we prove Theorem~\ref{teo-1}. In Section~\ref{grid-fam}
we introduce the grid family consisting of 
curves of bi-degree $(k,k)$ on a smooth quadric
which pass through a complete intersection
of type $(k,0),(0,k)$. We show that  the
grid family is exactly the subvariety  
of  $\mathcal F_k^{\rm tor}$ corresponding to 
torsion of order $k$ and 
has the expected dimension $4k-1$.
Section~\ref{density} is devoted to the proof of Theorem~
\ref{teo-2}. In Section~\ref{hyp} we prove a density
theorem for hyperelliptic curves, providing
an alternative proof for Theorem~\ref{teo-2}
in case $g=4$ (see Corollary~\ref{cor}). 
This result has an independent 
interest and is proved in the spirit of Griffiths 
computations of the infinitesimal invariant~\cite{Gr}.
Finally, in Section~\ref{exa} 
we consider examples
of curves $C$ with $\eta_C$ torsion.

In all the paper we work over the field of complex 
numbers except for Section~\ref{sym} (see Remark~\ref{ch}).


\section{The second symmetric product}
\label{sym}
Let $C$ be a smooth projective 
curve of genus $g>1$ defined over an algebraically 
closed field $\kk$ of characteristic $0$.
\begin{proposition}\label{pic}
The diagonal embedding 
$\imath\colon C\to C^{(2)}$
induces an isomorphism 
$\imath^*\colon \Pic^0(C^{(2)})\to
\Pic^0(C)$ of abelian varieties.
\end{proposition}
\begin{proof}
To prove the statement we explicitly construct
the inverse map of $\imath^*$. Given a point
$p\in C$ let $H_p$ be the curve of $C^{(2)}$
which is the image of $\{p\}\times C$ via $\pi$.
Define the map ${\rm Div}(C)\to {\rm Div}(C^{(2)})$ by
$\sum_in_ip_i\mapsto\sum_i n_iH_{p_i}$ and
observe that it maps principal divisors to principal 
divisors. The induced map of Picard groups  
restricts to a homomorphism $\Pic^0(C)\to
\Pic^0(C^{(2)})$ which is easily seen to be a 
right inverse of $\imath^*$.
Since the two abelian varieties
$\Pic^0(C)$ and $\Pic^0(C^{(2)})$ have 
the same dimension we conclude that
$\imath^*$ is an isomorphism.
\end{proof}
Observe that $\Delta$ is the branch divisor
of the double cover $\pi$ and thus 
its class is divisible by $2$ in $\Pic(C^{(2)})$.
Moreover the following linear equivalences
\[
 \Delta|_\Delta
 \sim
 -2K_\Delta
 \qquad
 \qquad
 K|_\Delta
 \sim
 3K_\Delta
\]
can be proved by passing to $C^2$ 
and calculating the restriction of $K_{C^2}$
to the diagonal.
By the Riemann-Hurwitz formula we get the equalities
$2(2g-2)^2 = K_{C^2}^2 = 2 (K + \frac{\Delta}{2})^2$ 
from which we deduce the following 
\begin{equation}
 \label{intersections}
 K^2 = (g-1)(4g-9)
 \qquad
 K\cdot\Delta = 6(g-1)
 \qquad
 \Delta^2 = -4(g-1).
\end{equation}
In particular the classes of $\Delta$ and $K$ are 
independent in the N\'eron-Severi group of ${C^{(2)}}$.
We let $\langle \Delta, K\rangle$ be the rational 
vector subspace of $\Pic({C^{(2)}})\otimes_\zz\qq$
generated by the classes of $\Delta$ and $K$
and form the following cone
\[
 {\rm Nef}(\Delta,K)
 =
 \{D\in\langle \Delta, K\rangle\, :\, D\text{ is nef}\}.
\]
This cone is related to the Kouvidakis conjecture
which predicts which ones are the extremal rays of 
${\rm Nef}(\Delta,K)$. In case the genus is a square,
i.e. $g = (k-1)^2$, the conjecture is known to be
true~\cite{CK} for a very general curve $C$ and it 
holds as well if $C$ has an irreducible $g_k^1$,
that is the curve~\eqref{eq:gamma} defined below
is irreducible.
In this case the extremal rays of ${\rm Nef}(\Delta,K)$ 
are spanned by the classes of $2K+3\Delta$ and 
$2K+(5-2k)\Delta$.

\begin{proposition}
\label{ray-1}
Let $C$ be a smooth curve of genus at least
two. Then the divisor 
$2K+3\Delta$ of $C^{(2)}$ is semiample.
\end{proposition}
\begin{proof}
Observe that the divisor $2K+3\Delta$
is big since $K$ is ample and $\Delta$
is effective. Moreover it is nef since 
$(2K+3\Delta)\cdot\Delta = 0$.
We have an exact sequence of sheaves
\[
 \xymatrix@1{
  0\ar[r]
  &
  \Osh_{C^{(2)}}(2K+2\Delta)\ar[r]
  &
  \Osh_{C^{(2)}}(2K+3\Delta)\ar[r]
  &
  \Osh_\Delta\ar[r]
  &
  0.
 }
\]
Since $2K+2\Delta = N + K+\frac{1}{2}\Delta$,
with $N=\frac{1}{2}(2K+3\Delta)$ nef and big,
then by the Kawamata-Viehweg vanishing
theorem and the long exact sequence in 
cohomology of the above sequence we 
conclude that $\Delta$ is not contained 
in the base locus of $|2K+3\Delta|$.
The statement follows by the ampleness 
of $K$ and the Zariski-Fujita theorem
~\cite[Remark 2.1.32]{La}.
\end{proof}

Assume now that $C$ is a smooth curve
of genus $g = (k-1)^2>1$ which admits a 
$g_k^1$ and define the following curve 
of $C^{(2)}$:
\begin{equation}
\label{eq:gamma}
 \Gamma \, :=\, \{p+q : g_k^1-p-q\geq 0\}.
\end{equation}
It can be easily proved that $\Gamma$ is irreducible 
if the $g_k^1$ does not contain a $g_r^1$ with $r<k$.
In particular this holds if the $g_k^1$ is simple.
Observe that the $g_k^1$ defines a morphism
$\Gamma\to\pp^1$ of degree $\frac{1}{2}k(k-1)$
whose branch points are exactly those
of the $g_k^1$. Thus if the $g_k^1$ is simple
we deduce that
\begin{equation}
\label{genus}
 2g(\Gamma)-2 = -k(k-1)+2(k-2)(k-1)k,
\end{equation}
where $g(\Gamma)$ is the genus of $\Gamma$.
By continuity the genus of $\Gamma$ stays the 
same even if the $g_k^1$ is not simple.

\begin{lemma}\label{gamma}
If $C$ admits a $g_k^1$ then the divisor
$2K+(5-2k)\Delta$ is numerically equivalent to 
$(4k-8)\Gamma$.
Moreover if $C\in\mathcal F_k$, the curves  
$\Gamma$ and $\Gamma'$ corresponding to the
two distinct $g_k^1$ on $C$ are disjoint.
\end{lemma}
\begin{proof}
Let $H$ be the curve of $C^{(2)}$ defined by
$\{p+q : q\in C\}$. By the Riemann-Hurwitz formula one
has the numerical equivalence
\[
 K
 \equiv
 (2g-2)H-\frac{1}{2}\Delta.
\]
Thus by the genus formula $2g(\Gamma)-2
= \Gamma^2+(2g-2)\Gamma\cdot H-\frac{1}{2}
\Delta\cdot\Gamma$, by Equation~\eqref{genus}
and the equalities $\Gamma\cdot H = k-1$,
$\Delta\cdot\Gamma = 2g-2+2k$ we deduce
$\Gamma^2=0$. Since the intersection matrix
of the divisors $\Gamma,K,\Delta$
\[
 \left(
  \begin{array}{rrr}
   0 & (2k-5)(k-1)k & 2(k-1)k\\
   (2k-5)(k-1)k & (2k-5)(k-2)k(2k+1) & 6(k-2)k\\
   2(k-1)k & 6(k-2)k & -4(k-2)k
  \end{array}
 \right)
\]
has rank two, by the Hodge Index theorem $\Gamma$ is
numerically equivalent to a rational linear
combination of $K$ and $\Delta$.
Moreover, being the kernel of the above matrix
generated by the vector $(4k - 8, -2, 2k - 5)$, we
get the first statement. 
Observe that  $\Gamma$ and $\Gamma'$  
have no common component since 
otherwise the map from $C$ to the quadric 
given by the two $g_k^1$ 
would not be an embedding.
Thus the second statement 
immediately follows from the equality
$\Gamma\cdot\Gamma' = \Gamma^2=0$.
\end{proof}

\begin{lemma}\label{fib}
If $f\colon C^{(2)}\to Y$ is a fibration with connected 
fibers onto a smooth curve, then $Y$ has genus 
at most one. If moreover a fiber of $f$ is numerically 
equivalent to a rational linear combination of 
$\Delta$ and $K$ then $Y$ is rational.
\end{lemma}
\begin{proof}
To prove the first statement observe that if
$Y$ had genus $\geq 2$, then there would
be two linearly independent holomorphic 
$1$-forms $w_1,w_2$ of $C^{(2)}$,   
obtained by pull-back of $1$-forms of $Y$, 
such that $w_1\wedge w_2 = 0$. On the other hand 
we have a commutative diagram of isomorphisms
\[
 \xymatrix{
  {\bigwedge^2 H^0(C^{(2)},\Omega^1)}
  \ar[rr]^-{w_1\wedge w_2\mapsto w_1\wedge w_2}
  \ar[rd]_\alpha &&
  H^0(C^{(2)},\Omega^2)\\
  & {\bigwedge^2 H^0(C,\Omega^1)}\ar[ru]_\beta 
 }
\]
where $\alpha$ is induced by pull-back via the 
diagonal embedding $C\to C^{(2)}$ and $\beta$
is defined by $w_1\wedge w_2\mapsto 
\frac{1}{4}(\pi_1^*w_1\wedge\pi_2^*w_2+
\pi_2^*w_1\wedge\pi_1^*w_2)$, with $\pi_i:C^2\to C$ 
the projections onto the two factors. 
Thus such $1$-forms cannot exist and $Y$
must have genus at most one.

To prove the second statement, assume
that a fiber $F$ of $f$ is numerically equivalent 
to a rational linear combination of 
$\Delta$ and $K$.
If $Y$ has genus one then we have the 
following commutative diagram
\[
 \xymatrix{
 C\ar[r]^-\imath\ar[d]^-{\alb_C} & C^{(2)}\ar[r]^-f\ar[d]^-{\alb_{C^{(2)}}} & Y\ar[d]^-{\alb_Y}_-\cong\\
 \Alb(C)\ar[r]^-{\alb_\imath}_-\cong & \Alb(C^{(2)})\ar[r]^-{\alb_f} & \Alb(Y)\\ 
 }
\]
where $\alb_\imath$ and $\alb_f$
are induced by the universal property of 
the Albanese morphism and $\alb_\imath$
is an isomorphism by Proposition~\ref{pic}.
Let $\Theta$ be the theta divisor of $\Alb(C^{(2)})$.
Recall that both $\alb_{C^{(2)}}(\Delta)$ and 
$\alb_{C^{(2)}}(K)$ are numerically equivalent to
a multiple of $\Theta^{g-1}$.
By our assumption on $F$ we deduce that
$\alb_{C^{(2)}}(F)\equiv c\,\Theta^{g-1}$. Moreover
$c$ is non-zero since $F$ cannot be 
contracted by $\alb_{C^{(2)}}$.
But this gives a contradiction, since 
$\alb_f\circ\alb_{C^{(2)}}$ contracts $F$, while
$\alb_f(\Theta^{g-1})$ cannot be a point, being the 
image of $\alpha_Y\circ f\circ\imath$.
\end{proof}

For a proof of the following lemma see also
~\cite{BHPV}*{Lemma III.8.3}.

\begin{lemma}
\label{covering}
Let $f\colon S\to\mathbb \cc$ be a proper 
morphism and let $nF$ be a multiple fiber 
of $f$ with multiplicity $n>1$. 
Then $\Osh_F(F)$ is torsion non-trivial.
\end{lemma}

\begin{proof}
We first show that $\mathcal L=\Osh_S(F)$ is not trivial.
Consider the closure of the graph of $f$ 
in $S\times \mathbb P^1$, let 
$\bar S$ be its minimal resolution and 
$\bar f:\bar S\to \mathbb P^1$ 
be the fibration given by the projection 
to the second factor. 
Assume by contradiction that $F\sim 0$ 
in $S$. Thus $F\sim \alpha F'$ in $\bar S$,
where $F'$ is a divisor with support contained 
in $\bar S-S=\bar f^{-1}(\infty)$ and $\alpha$ 
is a positive integer. 
Since $h^0(\bar S, nF)=2$ with $n>1$, then  
$h^0(\bar S, F)=1$, giving a contradiction.

The line bundle $\mathcal L$ thus defines a
non-trivial \'etale cyclic covering 
$\eta\colon S'\to S$.
By taking the Stein factorization of $f\circ\eta$
we get a commutative diagram
\[
 \xymatrix{
  S'\ar[r]^-\eta\ar[d]^-{f'} & S\ar[d]^-f\\
  B\ar[r]^-\nu & \cc,
 }
\]
where $f'$ is a morphism with connected 
fibers and $\nu$ is a finite map.
If $\mathcal L|_{F}=\Osh_F(F)$ is trivial, then 
$\nu$ is an
%
\'etale covering of $\mathbb \cc$, 
since the restriction of $\mathcal L$ 
to any fiber of $f$ is trivial.
Thus  $\nu$ is the trivial covering and $B$ 
has $n$ connected components,
a contradiction since $\eta$ is non-trivial.

Since $\mathcal L^{\otimes n}$ is trivial,
then clearly its restriction to $F$ is trivial.
This concludes the proof. 
\end{proof}

\begin{lemma}\label{triv}
Let $C$ be a non-hyperelliptic 
curve of genus $g = (k-1)^2>1$ 
which is complete intersection 
of a quadric cone with a degree $k$
surface of $\pp^3$. Then the divisor
$\Gamma$ of $C^{(2)}$, 
corresponding to the $g_k^1$ 
of $C$ defined by the ruling of 
the cone, is not semiample.
\end{lemma}
\begin{proof}
We first show that the line bundle
$\Osh_\Gamma(\Gamma)$ is trivial.
Indeed let $Q_t\subseteq\pp^3\times
\mathbb A^1$ be a family of quadrics
whose central fiber $Q_0$ is the cone
containing $C$ and whose general fiber 
is a smooth quadric. Let $\mathcal D$
be a divisor of $\pp^3\times\mathbb A^1$
which cuts out on the general fiber $Q_t$
a smooth curve $C_t$ of type $(k,k)$ with two
simple $g_k^1$ and $C$ on $Q_0$.
The family $\mathcal C\to\mathbb A^1$
of curves $C_t$ gives a family 
$\mathcal C^{(2)}\to\mathbb A^1$
whose general fiber is $C_t^{(2)}$.
On any such fiber there are two 
curves $\Gamma_t$, $\Gamma_t'$ 
corresponding to the two $g_k^1$
on $C_t$. The line bundle 
$\Osh_{\Gamma_t}(\Gamma_t')$ is
trivial, by Lemma~\ref{gamma},
and its limit is $\Osh_\Gamma(\Gamma)$,
which proves the claim.

Assume now, by contradiction, that 
$\Gamma$ is semiample. 
Since $\Gamma^2=0$,
a multiple $n\Gamma$ defines a morphism
$f\colon C^{(2)}\to B$, where $B$ is a curve. Moreover,
after possibly normalizing, we can assume $B$ 
to be smooth. Now, let $f=\nu\circ\varphi$ be 
the Stein factorization of $f$, where 
$\varphi\colon C^{(2)}\to Y$ is a morphism with
connected fibers. By Lemma~\ref{covering}
and the fact that $\Osh_\Gamma(\Gamma)$ is
trivial we deduce that $\Gamma$ is a union 
of fibers of $\varphi$. Moreover 
both the hypotheses of Lemma~\ref{fib} are satisfied, 
thus $Y$ must be a smooth rational curve.
Let $H$ be the curve of
$C^{(2)}$ which is the image of the curve 
$\{p\}\times C$ via $\pi$.
The equality $\Gamma\cdot H = k-1$ implies
that $\varphi|_H$ is a covering of $Y$ whose degree 
$d$ divides $k-1$. 
Thus $C\cong H$ would admit two maps to 
$\pp^1$ of degrees $k$ and $d$, respectively.
Being the degrees coprime, the curve $C$ 
would be birational to a curve of bi-degree $(k,d)$
of $\pp^1\times\pp^1$, whose genus is smaller
than $g$, a contradiction.
\end{proof}

 \begin{remark}\label{rem1}
 If $D$ is a prime divisor on a projective 
 surface $X$ such that $|D|$
 has dimension $0$ then $D$ is semiample 
 if and only if $\dim |nD| > 0$ for some $n>1$.
 Indeed the ``only if'' part is obvious, while 
 the other implication follows from the fact
 that the fixed divisor of $|nD|$ is $mD$ 
 for some $m<n$ and thus the base locus of
 $|(n-m)D|$ is at most zero-dimensional and
 one concludes by Zariski-Fujita theorem
 ~\cite[Remark 2.1.32]{La} 
 \end{remark}

\begin{proof}[Proof of Theorem~\ref{teo-1}]
We show that $(i)\Rightarrow (ii)$ holds.
Assume that the extended canonical ring
$R := R(\Delta,K)$ is finitely generated.
We begin to show that any divisor 
whose class is in $\Nef(\Delta,K)$ 
is semiample. By Proposition~\ref{ray-1}
it suffices to show that $2K+(5-2k)\Delta$
is semiample, or equivalently that the divisor
$\Gamma$ defined by a $g_k^1$ of $C$
is semiample, by Lemma~\ref{gamma}.
Let $f_1,\dots,f_r$ be a minimal system of 
homogeneous generators of $R$
with respect to the $\zz^2$-grading
and let $w_i := \deg(f_i)\in\zz^2$ be
the degree of $f_i$ for any $i$.
Let $f_1\in R$ be a generator whose 
degree belongs to the ray generated by 
$[\Gamma]$ and let $D$ be a
very ample divisor of the form $a\Gamma+b K$, with $a,b>0$
whose class $w$ lies in the interior of the cone

\begin{minipage}{0.5\textwidth}
\[
 \bigcap_{i=2}^r{\rm cone}(w_1,w_i)
 \cap {\rm cone}(w_1,[K])
 \]
\end{minipage}
\begin{minipage}{0.5\textwidth}
\begin{center}
 \begin{tikzpicture}[scale=0.65]
  \draw[-,thick] (0,0) -- (2,0) node[below]{$w_r$};
  \draw[-,thick] (0,0) -- (-2.2,1.1) node[left] {$w_1$};
  \draw[-,color=blue] (0,0) -- (0,1.8); 
  \node[right] at (-0.6,2.2){$[K]$};
  \foreach \x/\y in {-1/1.5,-1.5/1.3,-0.5/1.7} \draw[-,color=blue] (0,0) to (\x,\y);
  \foreach \x/\y in {1/1.5,1.5/1.3, 1.8/0.8} \draw[-,color=blue] (0,0) to (\x,\y);
  \fill[black] (0,0) circle (2pt);
  \fill[black] (2,0) circle (2pt);
    \fill[black] (-2.2,1.1) circle (2pt);
      \fill[blue] (-1,1.5) circle (2pt);
        \fill[blue] (-1.5,1.3) circle (2pt);
          \fill[blue] (-0.5,1.7) circle (2pt);
            \fill[blue] (1,1.5) circle (2pt);
              \fill[blue] (1.5,1.3) circle (2pt);
                \fill[blue] (1.8,0.8) circle (2pt);
                  \fill[blue] (0,1.8) circle (2pt);
\node[above] at (-1,1.5) {};
  \node[above, red] at (-1.7,1.3) {$w$};
   \node[above] at (-1,1.5) {$w_2$};
     \node[above] at (1.7,1.3) {$w_i$};
\end{tikzpicture}
\end{center}
\end{minipage}

\noindent If $\Gamma$ is not semiample, then
any section of $R_w = H^0(C^{(2)},D)$
is divisible by $f_1$, a contradiction.
Thus we showed that $\Gamma$ is semiample
and as a consequence of Lemma~\ref{triv}
the curve $C$ has two $g_k^1$. We denote
the corresponding curves of $C^{(2)}$ by 
$\Gamma$ and $\Gamma'$.
A multiple of $\Gamma$
defines a morphism $f\colon C^{(2)}\to B$
onto a smooth curve $B$ whose Stein factorization is 
the following 
\[
 \xymatrix{
 C^{(2)}\ar[r]^{\varphi}\ar[rd]_{f} & Y\ar[d]\\
 & B.
 }
\]
Two fibers of $\varphi$ are $n\Gamma$ and $m\Gamma'$
for some positive rational numbers $n,m$. Since
$\Gamma$ is numerically equivalent to $\Gamma'$ 
by Lemma~\ref{gamma}, then $n=m$. 
Moreover by Lemma~\ref{fib} the curve $Y$
is rational, so that $n\Gamma$ is linearly equivalent
to $n\Gamma'$. This implies that the class
of $\Gamma-\Gamma'$ is torsion non-trivial
in $\Pic^0(C^{(2)})$
and thus the same holds for 
\begin{align*}
 2\eta_C 
 & =
 \text{ramification of $g_k^1$ - ramification of ${g_k^1}'$}\\
 & =
 \imath^*(\Gamma-\Gamma').
\end{align*}
We now show that $(ii)\Rightarrow (i)$ holds.
First of all observe that if $\eta_C$ is torsion
non-trivial, then the same holds for
$\Gamma-\Gamma'$ by the above equalities
and the fact that $\imath^*$ is an isomorphism.
In this case $n\Gamma\sim n\Gamma'$ for
some positive integer $n$ and this implies
that $|n\Gamma|$ is base point free, being
$\Gamma^2=0$. Thus $\Gamma$ is semiample.
Hence, if $L\subseteq\zz^2$ is the submonoid
generated by the integer points of the cone 
${\rm cone}(2K+3\Delta,2K+(5-2k)\Delta)$,
the following subalgebra 
\[
  S := \bigoplus_{(a,b)\in L}H^0(C^{(2)},a\Delta+bK)
\]
of $R$ is finitely generated 
by~\cite[Lemma 4.3.3.4]{ADHL}.
The homogeneous elements of $R$ 
which do not belong to $S$ are sections
of Riemann-Roch spaces 
$H^0(C^{(2)},D)$ with $D\cdot\Delta < 0$.
Hence any such section is divisible by
the generator $f_\Delta$ of $R$ which is a defining
section for $\Delta$. Thus we conclude that 
$R$ is generated by $S$ and $f_\Delta$
and the statement follows.
\end{proof}

\begin{remark}
\label{ch}
Over an algebraically closed field $\kk$ of positive 
characteristic the statement of Theorem~\ref{teo-1}
must be modified as follows: the algebra
$R(\Delta,K)$ is finitely generated if and 
only if $\eta_C$ is torsion. Indeed if $\eta_C$ 
is trivial, then $\Osh_\Gamma(\Gamma)$ is trivial
as well as shown in the proof of Lemma~\ref{triv},
thus $\Gamma$ is semiample by~\cite[Theorem 0.2]{Ke}.
Moreover, if $\kk$ is the algebraic closure
of a finite field, then $\eta_C$ is always
torsion and thus $R(\Delta,K)$ is always 
finitely generated.

The conclusion of Lemma~\ref{covering}
is no longer true in positive characteristic. Indeed in
this case the algebraic fundamental group of the
affine line $\kk$ is not trivial and thus there is no
contradiction. For example, in characteristic $p>0$,
if $C$ is the curve of $\kk^2$ defined by the equation 
$x_1^p-x_1=x_2$, then the projection onto the second
factor defines a non-trivial \'etale covering of $\kk$.
\end{remark}

\begin{remark}
We observe that the locus of smooth curves of genus
$(k-1)^2 > 1$ which admit two $g_k^1$ has one component
of maximal dimension which consists of curves of 
type $(k,k)$ on a smooth quadric $Q$. 
Indeed by~\cite{AC} the only
other component which could have bigger dimension 
would consist of curves $C$ admitting an involution. 
A parameter count shows that for our curves this 
component has smaller dimension.

Moreover any smooth curve $C$ of type
$(k,k)$ on $Q$ admits exactly two $g_k^1$.
Indeed let $S = \{p_1,\dots,p_k\}$ be a set of
$k$ distinct points with $p_1+\dots+p_k\in g_k^1$. 
By the Riemann-Roch theorem $S$ is in Cayley-Bacharach 
configuration with respect to the curves of 
type $(k-2,k-2)$. It follows that all the 
points of $S$ are collinear. Indeed, let
$\ell$ be the line through the first two points
$p_1,p_2$, let $H$ be a general hyperplane 
which contains $\ell$ and let $q\in S\setminus\{p_1,p_2\}$.
Take a union $\Lambda$ of $k-3$
hyperplanes through all the points of 
$S\setminus\{p_1,p_2,q\}$ and such that
$q\notin\Lambda$.
Then $H\cup \Lambda$ cuts out on $Q$ 
a curve of type
$(k-2,k-2)$ which, by the Cayley-Bacharach
configuration, must pass through $q$.
Hence $q\in H$ and by the generality
assumption on $H$ we deduce $q\in\ell$.
\end{remark}


\section{The grid family}
\label{grid-fam}

In this section we study families of curves of type $(k,k)$
on a smooth quadric $Q = \pp^1\times\pp^1$
such that the class of the difference of the two
$g_k^1$ induced by the two rulings is 
$n$-torsion. We prove that $n\geq k$
and construct the family with $n=k$.

\begin{definition}
\label{grid}
Given two effective 
divisors $L_1$ and $L_2$ of $Q$ of type
$(k,0)$ and $(0,k)$ respectively, the {\em grid 
linear system} defined by $L_1$ and $L_2$ is 
the linear system of curves of $Q$ of 
bi-degree $(k,k)$ which pass through the 
complete intersection $L_1\cap L_2$.
The {\em grid family} 
\[
 \mathcal G_k
 \subseteq\mathcal F_k
\]
is the family 
of all curves of type $(k,k)$ which belong to some 
grid linear system.
\end{definition}
Observe that if $C$ is a smooth curve in $\mathcal G_k$
and $\eta_C\in \Pic^0(C)$ is the class of the difference of the 
two $g_k^1$ cut out by the two rulings then $\eta_C^{\otimes k}$
is trivial. This justifies the inclusion
$\mathcal G_k\subseteq\mathcal F_k^{\rm tor}$.
Any curve $C$ in $\mathcal G_k$
admits an equation of the form
\begin{equation}
\label{equ-fam}
 f_1(x_0,x_1)g_2(y_0,y_1)+g_1(x_0,x_1)f_2(y_0,y_1)
 =
 0,
\end{equation}
where $f_1, f_2, g_1$ and $g_2$ are homogeneous
polynomials of degree $k$. Indeed it is enough to
prove the claim for a curve $C$ in a grid family
where both $L_1$ and $L_2$ consist of
distinct lines and then conclude by specialization 
that the same holds for any
curve $C$ of $\mathcal G_k$.
Let $h=0$ be an equation for $C$ and
let $f_i = 0$ be an equation for $L_i$, for $i=1,2$.
By the equality
\[
 V(h,f_1) = V(f_1,f_2),
\]
the fact that the two ideals $(h,f_1)$ and $(f_1,f_2)$ are both
radical and saturated with respect to the
irrelevant ideal $(x_0,x_1)\cap (y_0,y_1)$ of $Q$,
we deduce that the equality $(h,f_1) = (f_1,f_2)$
holds. The claim follows since $h$ has bi-degree $(k,k)$.
Observe that the general element of $\mathcal G_k$ is
irreducible and smooth.

\begin{proposition}
The {\em grid family} $\mathcal G_k$
has dimension $4k-1$.
\end{proposition}
\begin{proof}

The projectivization of the set of homogeneous polynomials of
bidegree $(k,k)$ of type $f(x_0,x_1)g(y_0,y_1)$ can be identified 
with the image $\mathcal S$ of the Segre embedding of 
$\pp^k\times\pp^k\to\pp^N$,
where $N={k^2+2k}$.
Thus, a curve in $\mathcal G_k$ can be identified 
with a point of the $1$-secant variety ${\rm Sec}(\mathcal S)$
of $\mathcal S$.
By~\cite[Theorem 1.4]{CS} the dimension of
${\rm Sec}(\mathcal S)$ is $4k-1$ and the 
statement follows.
\end{proof}

\begin{lemma}
\label{composed}
Let $C\in\mathcal F_k$ and let $D$ be a divisor
of $C$ cut out by one of the two rulings. Then,
for any $1\leq n\leq k-1$, the linear system $|nD|$ 
is composed with the pencil $|D|$.
\end{lemma}
\begin{proof}
To prove the statement it is enough to show that
$h^0(C,nD) = n+1$ for $1\leq n\leq k-1$. 
Let $Q = \pp^1\times\pp^1$. Without loss of generality
we can assume that $D$ is cut out by the first ruling
of $Q$, so that we have the following exact sequence 
of sheaves
\[
 \xymatrix@1{
  0\ar[r] 
  &
  \Osh_Q(-k+n,-k)\ar[r]
  &
  \Osh_Q(n,0)\ar[r]
  &
  \Osh_C(nD)\ar[r]
  &
  0.
 }
\]
Taking cohomology, using the vanishing of the higher
cohomology groups of the middle sheaf, the Serre's 
duality theorem and the hypothesis on $n$ we deduce 
the following equalities: 
\[
h^1(C,nD) = h^0(Q,\Osh_Q(k-n-2,k-2))= (k-n-1)(k-1).
\]
By the adjunction formula $C$ has genus $(k-1)^2$.
Thus by the above and the Riemann-Roch formula 
we conclude
\[
 h^0(C,nD)
 =
 nk+1-(k-1)^2+(k-n-1)(k-1)
 =
 n+1
\]
and the statement follows.
\end{proof}

\begin{proposition}
\label{torsion}
Let $C$ be a smooth curve of  type
$(k,k)$ on $Q$ and let $\eta_C$
be the class of the difference of the two $g_k^1$. 
Then $\eta_C$ has order $\geq k$ and the 
following are equivalent:
\begin{enumerate}
\item
$\eta_C$ has order $k$;
\item
the curve $C$ belongs to the grid family $\mathcal G_k$.
\end{enumerate}
\end{proposition}
\begin{proof}

We first show that $\eta_C$ has order $\geq k$.
Let $D_1$ be a divisor in the first $g_k^1$ and let  
$D_2$ be a divisor in the second $g_k^1$. 
Let $n<k$ be a positive integer.
By Lemma~\ref{composed} the linear system
$|nD_i|$ is composed with $|D_i|$ for any $i$. 
Thus the linear equivalence $nD_1\sim nD_2$ would 
imply that for any set of $n$ lines of the first ruling 
there are $n$ lines of the second ruling which cut 
out the same set of points on $C$.  This cannot be
since such set of $nk$ points would lie on a grid 
with $n^2$ points.
This proves the claim.

We now show that $(i)\Rightarrow (ii)$ holds.
By the exact sequence of sheaves
\[
 \xymatrix@1{
  0\ar[r] 
  &
  \Osh_Q(0,-k)\ar[r]
  &
  \Osh_Q(k,0)\ar[r]
  &
  \Osh_C(kD_1)\ar[r]
  &
  0
 }
\]
the fact that $h^0 = 0$, $h^1=k-1$ for the first sheaf
and $h^0 = k+1$, $h^1 = 0$ for the second sheaf  we
get the equality $h^0(C,\Osh_C(kD_1)) = 2k$. 
By hypothesis $|kD_1| = |kD_2|$ holds and by the above 
calculation the linear system is a projective space of dimension
$2k-1$. Let $\mathcal H_i\subseteq |kD_1|\cong\pp^{2k-1}$ 
be the projectivization of the $k$-th symmetric
power of the vector space $H^0(C,D_i)$.
Since both $\mathcal H_1$ and $\mathcal H_2$
have dimension $k$, their intersection 
$\mathcal H_1\cap\mathcal H_2$ is at least
one-dimensional. A point of this intersection 
corresponds to a divisor $D$ which is cut out 
on $C$ by $k$ lines of the first ruling
and by $k$ lines of the second ruling. 
This proves the claim.

The implication $(ii)\Rightarrow (i)$ is obvious.
\end{proof}


\section{Density of the torsion locus}
\label{density}

We define $\mathcal F_k$ 
to be the open subset of the Hilbert scheme of the smooth
curves of type $(k,k)$ of $\pp^1\times\pp^1$.
Let $\pi\colon \mathcal C\to\mathcal F_k$ be the universal
family, let $\mathcal H = R^1\pi_*\cc$ be the associated
Hodge bundle whose fiber
over a point $C\in\mathcal F_k$ is the
cohomology group $H^1(C,\cc)$ and
let $\mathcal H^{1,0}$ be the subbundle 
of $\mathcal H$ whose fiber over
$C$ is $H^{1,0}(C)$. We recall that
the jacobian family $\mathcal J\to\mathcal F_k$
can be defined as 
\begin{equation}
\label{jac}
 \mathcal J
 =
 \frac{\mathcal H}{\mathcal H^{1,0}+R^1\pi_*\zz}.
\end{equation}
Given a point $t\in\mathcal F_k$ we denote by
$C_t$ the fiber of $\pi$ over $t$ and 
by $\nu_t = \mathcal L_1\otimes\mathcal L_2^{-1}
\in J(C_t)$ the class of the difference 
of the two $g_k^1$ cut out by the two rulings.
This gives a normal function $\nu\colon\mathcal F_k
\to \mathcal J$ defined by $t\mapsto\nu_t$.
We now consider a $\mathcal C^\infty$ trivialization 
$\varphi$
of the jacobian family over an open, simply connected
subset $\mathcal U$ of $\mathcal F_k$:
\[
 \xymatrix{
 \mathcal J
  \ar[r]^-\varphi & 
  \mathcal U\times\dfrac {R^1\pi_\ast\mathbb R} {R^1\pi_\ast \mathbb Z}
  = 
  \mathcal U\times\mathbb T,
 }
\]
where $\mathbb T$ is the real torus 
$\rr^{2g}/\zz^{2g}$.
We let ${\rm pr}_2\colon\mathcal U\times\mathbb T
\to\mathbb T$ be the projection on the second factor
and let 
\[
 \eta\colon{\mathcal  U}\to\mathbb T
 \qquad
 t\mapsto {\rm pr}_2(\varphi(\nu(t))).
\]
We will show that, for any $k\geq 3$, the image 
of $\eta$ contains a non-empty open subset.
Our strategy will be to prove that 
the fibers of $\eta$ are complex subvarieties 
of $\mathcal U$, that $\eta$ has a fiber 
of the expected real dimension 
$8k-2 = \dim \mathcal F_k - \dim\mathbb T$
and we conclude by Proposition~\ref{fiber}
(see also the argument in the proof of~\cite[Proposition 3.4]{CPP}).

\begin{proposition}
\label{fiber}
Let $V\subseteq X$ be an inclusion of complex manifolds, 
$Y$ be a real manifold and 
$\eta\colon X\to Y$ be a $\mathcal C^\infty$
map whose fibers are complex subvarieties
of $X$. Assume that there is a point $p\in V$
such that the following hold:
\begin{enumerate}
\item
the differential $d\eta_p$ is surjective,
\item
the fiber of the restriction $\eta|_V$ over $\eta(p)$
has dimension $\dim V - \dim Y\geq 0$.
\end{enumerate}
Then $\eta(V)$ contains
a non-empty open subset of $Y$.
\end{proposition}
\begin{proof}
Let $F$ be the fiber of $\eta$ over $\eta(p)$ 
and let $U$ be a coordinate neighbourhood of $p$
such that the restriction $\eta|_U$ is a submersion.
Let $d = \dim V - \dim Y$. After possibly 
intersecting $V$ with $d$ general smooth 
complex hypersurfaces passing through $p$
we can reduce to the case $d = 0$.
Thus, after possibly shrinking $U$ 
we can assume $F\cap V\cap U = \{p\}$.
For $q$ in a sufficiently small neighbourhood
$W$ of $\eta(p)$ the intersection number 
$\eta^{-1}(q)\cap V$ does not change
by~\cite[Pag.\,664]{GH}, so that the intersection is
non-empty. Then $\eta(V)$ contains $W$
and the statement follows.
\end{proof}

\begin{proposition}
The fibers of the map $\eta$ are complex varieties.
\end{proposition}
\begin{proof}
We restrict the Hodge and the jacobian bundle 
to the simply connected open subset $\mathcal U$
of $\mathcal F_k$,
where the family is topologically trivial.
Let $t_0\in \mathcal U$ be a distinguished  point
and let $C = C_{t_0}$ be the corresponding
curve. We construct the following
commutative diagram
\[
 \xymatrix{
  \mathcal U\ar@{=}[d]\ar[r]^-{\tilde\nu}
  & \mathcal H_{\mathcal U}\ar[r]^-\cong\ar[d]
  & \mathcal U\times H^1(C,\cc)\ar[d]\ar[r]
  & H^1(C,\cc)\ar[d]\\
  \mathcal U\ar[r]^-{\nu|_\mathcal U}
  & \mathcal J_{\mathcal U}\ar[r]^-\varphi
  & \mathcal U\times\mathbb T\ar[r]
  & \mathbb T
 }
\]
where $\tilde\nu$ is a lifting of the restriction
$\nu|_{\mathcal U}$ and the maps in the top row
are complex analytic morphisms.
Observe that, given $x\in \mathbb T$, we have 
\[
\eta^{-1}(x)=\{t\in\mathcal U: \tilde \nu(t)-\tilde x+c\in\mathcal H^{1,0}\},
\]
where the projection of $\tilde x$ to $\mathcal J$ is $x$ 
and $c\in H^1(C,\zz)$.
This implies that $\eta^{-1}(x)$ 
is a complex subvariety of $\mathcal U$ 
since $\mathcal H^{1,0}$ is a 
complex subbundle of $\mathcal H$.
\end{proof}

\begin{proof}[Proof of Theorem~\ref{teo-2}]
We will start proving that the image of the map $\eta$ 
contains a non-empty open subset of $\mathbb T$.
We consider the map
\[
 \phi\colon\mathcal J_{\mathcal U}\to \mathbb T
\]
which is composition of the trivialization $\varphi$
with the projection onto the second factor.
Observe that the differential of $\phi$ has
maximal rank at any point. By applying
Proposition~\ref{fiber} to the map
$\phi$, the subvariety 
$V = \nu|_{\mathcal U}(\mathcal U)$ of $\mathcal J$
and the fiber $\mathcal G_k$ of the restriction
$\phi|_V$ we deduce that $\eta(V)$ contains
a non-empty open subset.
Observe that the trivialization $\varphi$
can be taken real analytic, since it is
obtained by taking the real part of the 
complex analytic isomorphism coming
from the trivialization of the Hodge bundle.
By Sard theorem the set of non-critical values
of $\phi$ is an open subset in the analytic 
topology. Since $\phi$ admits an analytic continuation
on any simple arc in $\mathcal J$, its set
of non-critical points is dense in the analytic
topology.
We conclude by observing that the set
of torsion points of $\mathbb T$ is dense
and $\mathcal F_k^{\rm tor}$ is the preimage
of this set via $\eta = \phi\circ\nu$.
\end{proof}


\section{Hyperelliptic curves}
\label{hyp}

In this section we prove a density
theorem for hyperelliptic curves. This result 
has an independent interest and is proved in 
the spirit of Griffiths computations of the 
infinitesimal invariant~\cite{Gr}.
As an application to our problem, this result
provides an alternative proof for Theorem~\ref{teo-2}
in case $g=4$.

Let $\pi\colon\mathcal C\to\mathcal U$ be a versal family 
of hyperelliptic curves of genus $g > 1$, where $\mathcal U$
is simply connected of dimension $2g-1$.
We denote by $j\in\Aut(\mathcal C)$ the hyperelliptic
involution and by
$\mathcal J\to\mathcal U$ the jacobian
family. We consider the Abel-Jacobi map 
\[
 \nu\colon \mathcal C \to \mathcal J,\quad  
 x\mapsto\int_{j(x)}^x.
\]
If we take $\pi^\ast \mathcal J\to \mathcal C$ the pull-back 
of the Jacobian family on $\mathcal C,$ 
we may consider $\nu$ as a normal function.
As in Section~\ref{density} we consider a $\mathcal C^\infty$
trivialization of the jacobian family
$\mathcal J\cong \mathcal U\times \mathbb T$,
where $\mathbb T\cong J(C)$, to construct a map
\[
 \gamma\colon\mathcal C\to\mathbb T.
\]
\begin{theorem}
\label{diff}
If $p\in C$ is not a Weierstrass point, then the
differential of $\gamma$ at $p$ is surjective.
\end{theorem}

Our strategy is as follows. For any holomorphic
form $\omega\in H^0(C,\Omega_C)$ we show that
there is a curve $r(t)$ in $\mathcal C$ such that
$r(0) = p$ and $d\gamma_p(r'(0))\cdot\omega$
is non-zero. To this aim we produce $r(t)$
accordingly to the order $n$ of vanishing of $w$ 
at $p$. Since the divisor ${\rm div}(w)$ is $j$-invariant, 
then it is natural to consider $D = p+j(p)$.
We thus have a filtration 
\begin{equation}
\label{filt}
 H^0(C,\Omega_C)
 =
 L^0\supseteq L^1\supseteq...\supseteq L^{g-1}\supseteq L^g=0,
\end{equation}
where $L^k$ is the Riemann-Roch space 
$H^0(C,\Omega_C(-kD))$. Given $\omega\in L^k
\setminus L^{k+1}$, with $k>0$, we construct $\zeta = \partial(f)
\in H^1(C,T_C)^j$ as in Subsection~\ref{def}, where $f\in
H^0(Z,\Osh_Z)$. The one-dimensional family
\[
 \mathcal C_\zeta\to\Delta
\]
defined by $\zeta$, plus a choice of a smooth
section through the point $p$ in the family, defines a curve 
$r(t)$ in $\mathcal C$. We show that the family
is equipped with a $\mathcal C^\infty$
 $1$-form $\Theta$ such that the restriction 
of the $(1,0)$-part $\Theta^{1,0}$
to the central fiber $C$
admits a local expansion at $p$ of the
form $w + \tilde f(z)dt + o(t)$, where $z$ is a coordinate
in $C$, $\tilde f|_Z=f$ and $t\in\Delta$. We finally prove 
\[
 d\gamma_p(r'(0))\cdot\omega
 =
 \lim_{t\to 0}\frac{1}{t}\left(\int_{\Gamma_t}\Theta_t-\int_{\Gamma_0}\omega\right)
 =
 2f(p)\neq 0,
\]
where $\Gamma_t$ is a path between $r(t)$ 
and $j(r(t))$. When $k=0$ we choose $r(t)$
to be a path within the central fiber $C$,
we write $\omega$ locally as $h(z)dz$
and prove the equality
\[
 d\gamma_p(r'(0))\cdot\omega
 =
 \lim_{t\to 0}\frac{1}{t}\int_{p}^{r(t)}\omega
 =
 h(0)\neq 0.
\]

\subsection{Deformation of curves}
\label{def}
We recall first a result on deformation and on 
extension of line bundles 
which will be applied to hyperelliptic curves (see also~\cite{CP} and~\cite{R}).
Let $C$ be a smooth curve of genus $g>1$ and 
let $T_C$ and $\Omega_C$ be respectively the holomorphic 
tangent bundle and the canonical line bundle of $C$. 
Fix a non trivial $\omega\in H^0(C,\Omega_C)$ and 
let $Z$ be the canonical divisor associated to $\omega.$
The form $\omega$ defines the following exact sequence
\[
\xymatrix@1{
0\ar[r] & T_C\ar[r]^{\omega}& \Osh_C\ar[r]& \Osh_Z\ar[r]& 0.
}
\]
Passing to the long exact sequence in cohomology 
we obtain
\[
\xymatrix@C=20pt{
0\ar[r]& \cc\cong H^0(C,\Osh_C)\ar[r]& 
H^0(Z,\Osh_Z)\ar[r]^\partial& H^1(C,T_C)
\ar[r]^{\omega}& H^1(C,\Osh_C).
}
\]
Given an element $f\in H^0(Z,\Osh_Z)$
its image $\zeta=\partial(f)$ defines an
extension of $\Osh_C$ by $T_C$ via the
isomorphism $H^1(C,T_C)\cong 
{\rm Ext}^1(\Osh_C,T_C)$
\[
\xymatrix@1{
0\ar[r] & T_C \ar[r] & E_{\zeta}\ar[r] & \Osh_C\ar[r] & 0.
}
\]
Taking tensor product with $\Omega_C$ 
and recalling that $T_C$ is dual with $\Omega_C$
we get the following exact sequence
\begin{equation} \label{zeta}
\xymatrix@1{
0\ar[r] &\Osh_C\ar[r] & E_{\zeta}\otimes \Omega_C\ar[r] &
 \Omega_C\ar[r]& 0,
 }
 \end{equation}
which passing to the long exact sequence in cohomology
gives the following sequence 
whose coboundary  
is the cup product with $\zeta$:
\[
 \xymatrix@C=20pt{
0\ar[r] & \cc\cong H^0(C,\Osh_C)\ar[r]&
 H^0(C,E_\zeta\otimes \Omega_C)\ar[r]^-k & 
 H^0(C,\Omega_C)\ar[r]^-{\zeta}\ar[r]&
  H^1(C,\Osh_C).
  }
\]
Since $\zeta\in\ker(\omega)$, or equivalently
the cup product $\zeta\cdot \omega$ vanishes,
there exists an element $\Omega\in 
H^0(C,E_\zeta\otimes \Omega_C)$ such that 
$k(\Omega)=\omega$.
Now we consider the commutative diagram

\[
 \xymatrix@1{
  & 0\ar[d] & 0\ar[d] & H^0(C,\Osh_C)\ar[d]^-\zeta\\
  0\ar[r]\ar[d]
  & H^0(C,\Osh_C)\ar[r]\ar[d] 
  & H^0(Z,\Osh_Z)\ar[d]^-\rho\ar[r]^-\partial 
  & H^1(C,T_C)\ar[d] \\
  0\ar[r]\ar[d]
  & H^0(C,E_\zeta\otimes \Omega_C)\ar[r]^-r\ar[d]^-k
  & H^0(Z,E_\zeta\otimes \Omega_C|_Z)\ar[d]\ar[r]
  & H^1(C,E_\zeta)\ar[d]\\
  H^0(C,\Osh_C)\ar[r]^-\omega
  & H^0(C,\Omega_C)\ar[r]
  & H^0(Z,\Omega_C|_Z)\ar[r]
  & H^1(C,\Osh_C)
 }
\]
 \vspace{0.2cm}
 
The restriction of the lifting $\Omega$ to $Z$ 
gives an element $r(\Omega)\in H^0(Z,E_\zeta\otimes \Omega_C|_Z)$ 
that by construction is in the image of the map
$\rho$, that is $\Omega=\rho(g)$ 
for some $g\in H^0(Z,\Osh_Z)$.
A diagram chase proves indeed that
$\partial (g)=\zeta$ holds.
Since $\ker(\partial)$ is isomorphic to $\cc$,
we conclude that $f$ equals $g$ up to a constant.
This means that we can realize the function $f$ 
by a unique suitable lifting $\Omega$ of $\omega.$
We collect the discussion in the following lemma.

\begin{proposition}\label{coord}
Let $\omega$ be a non-zero element in  
$H^0(C,\Omega_C)$ and let $Z$ be the 
divisor of $\omega$. 
Then for any $f\in H^0(Z,\Osh_Z)$ 
there is a unique 
$\Omega\in H^0(C,E_{\zeta}\otimes \Omega_C)$ 
such that $r(\Omega)=\rho (f)$. 
\end{proposition}

When we interpret $H^1(C,T_C)$ 
as the space of first order deformations of $C$,
so that $\zeta$ corresponds to a family
\[
\mathcal C_\zeta\to {\rm Spec}\, \cc[\varepsilon],
\]
the isomorphism $H^1(C,T_C)\to
{\rm Ext}^1(\Osh_C,\Omega_C)$
gives the identifications 
$E_{\zeta}\cong T_{\mathcal C}|_{C}$ 
and $E_{\zeta}\otimes \Omega_C\cong\Omega_{\mathcal C}|_{C}.$ 
Therefore the sequence~\eqref{zeta}
is the cotangent sequence of the first order deformation.
In coordinates we may write 
\begin{equation}
\label{Omega}
 \omega= h(z)dz,\qquad \Omega=h(z)dz+\tilde f(z)dt,
\end{equation}
where $dt$ is the global section 
of the cotangent $\Omega_{\mathcal C}|_{C}$ 
and $f=\tilde f|_Z$ is a section of $H^0(Z,\Osh_Z)$ 
such that $r(\Omega)=\rho (f)$.

\subsection{The normal function}
In this subsection we will specialize the previous
construction to hyperelliptic families.
First of all, given a hyperelliptic curve $C$ of genus $g>1$,
we consider the $j$-invariant subspace
\[
 H^1(C,T_C)^{j}
 \subseteq
 H^1(C,T_C),
\]
which corresponds to the directions that are preserved
by the hyperelliptic involution $j$.
Since $j$ acts on $H^0(Z,\Osh_Z)$
as $f\mapsto f\circ j$
and it acts as $-1$ on $H^0(C,\Omega_C)$, 
then $\partial(f) = \zeta$ is $j$-invariant if and only if
$f\circ j$ equals $-f$ up to a constant.

Let $p$ be a non-Weierstrass point of $C$ and
$\omega\in H^0(C,\Omega_C)$
be a holomorphic form which vanishes 
with order $k>0$ at $p$,
that is $\omega\in L^k\setminus L^{k+1}$.
We now take $f_{\omega}\in H^0(Z,\Osh_Z)$, where $Z 
= k(p+j(p)) + Z' = {\rm div}(\omega)$,
such that
\[
 f_{\omega}(p)=1,\quad f_{\omega}(j(p))=-1,\quad f_{\omega}(p')=0 \text{ for } p'\in Z'.
\]
Given $\zeta_{\omega}=\partial(f_{\omega})$,
by the previous remark
we have $j(\zeta_{\omega})=\zeta_{\omega}.$ 
Thus there exists  a smooth family of hyperelliptic curves
\begin{equation}
 \pi\colon \mathcal C_\omega\to \Delta \label{curvamod},
\end{equation} 
 where $\Delta\subset \mathcal U$ is a disk, such that 
$\pi^{-1}(0)=C$ and such that 
the Kodaira-Spencer 
class of the family is $\zeta_{\omega}$. 
Let $\Omega_{\omega}$ be the section of 
$E_{\zeta_{\omega}}\otimes \Omega_C$ 
associated to $f_{\omega}$ as in Lemma ~\ref{coord}.
 By means of a trivialization of 
 the family we can construct a closed 
 differential $1$-form $\Theta$
 on $\mathcal C_{\omega}$ which is invariant 
 with respect to the involution $j$ and
 such that the restriction of the $(1,0)$-part
 $\Theta^{(1,0)}$ to the central fiber equals 
 $\Omega_{\omega}$. 
In local coordinates we can write  
  \[
  \Theta(z,t)^{(1,0)}=\Omega_{\omega}+ o(t)=\omega+f_{\omega}(z)dt+o(t).
  \]

We also assume to have a holomorphic section $r$ 
of $\pi_{\omega}$ such that $r(0)=p$ and define $r'= j(r).$
Moreover, we fix a differentiable map
\[
\Gamma(t,s):\Delta\times [0,1]\to \mathcal C_{\omega}
\] 
such that $\Gamma(t,s)\in \pi^{-1}(t)$,
 $\Gamma(t,0)=r'(t)$ and  $\Gamma(t,1)=r(t).$ 
 Observe that $\Gamma$ is a family of sections 
 connecting $r'$ and $r$.
 We define the function
\[
 g\colon\Delta\to \cc,\quad 
 t\mapsto 
 \int_{\Gamma_t} \Theta_{t}.
\]
 Following Griffiths~\cite[(6.6)]{Gr} and using the fact that 
 the Gauss-Manin connection vanishes on $\Theta$,
 we have that the derivative of $g$ at $0$ 
 equals
$
d\gamma_p(r'(0))\cdot \omega.
$

\begin{proof}[Proof of Theorem~\ref{diff}]
With the previous notation, 
given any $\omega\in H^0(C,\Omega_C)$ 
vanishing of order $k>0$ at $p$, 
we consider the family $\pi_{\omega}$, 
with its sections $r=r_{\omega}$ and $r'=j(r)$, 
and $\Theta$ the corresponding differential form on 
$\mathcal C_{\omega}$.
By the previous remark we have that 
\[
d\gamma_p(r'(0))\cdot \omega=g'(0).
\]
We now compute the latter term:
\[
g(t)-g(0)=\int_{\Gamma_t} \Theta_{t}-\int_{\Gamma_0} 
\Theta_{0}=\int_{\Gamma_t} \Theta_{t}-\int_{\Gamma_0} \omega.
\]
We call $r_t$ and $r'_t$ 
the arcs $r([0,t])$ and $r'([0,t])$ respectively.
Since $\Theta$ is closed, then $\Gamma^\ast(\Theta)$ is
exact and we have 
$0= \int_{r'_t}\Theta_t+\int_{\Gamma_t}\Theta_t - \int_{r_t}\Theta_t-\int_{\Gamma_0}\Theta_t$, 
hence
\[
g(t)-g(0)=  \int_{r_t}\Theta_t- \int_{r'_t}\Theta_t=2\int_{r_t}\Theta_t,
\]  
where the last equality is due to the fact 
that $j^\ast(\Theta)=-\Theta.$
Finally, since $\Theta = \omega+f_{\omega}(z)dt + o(t)$,
by the fundamental theorem of calculus 
we get
\[
 \lim_{t\to 0}\frac{1}{t}\int_{r_t}\Theta_t
 =
 \lim_{t\to 0}\frac{1}{t}\int_{r_t}\Theta_t^{(1,0)}
 =
 f_{\omega}(p)\not=0.
\]
Thus $g'(0)\not=0$.
If $k=0$, that is $\omega$ does not vanish 
at $p$, we will choose a loop $r(t)$ in $C$ with
$r(0)=p$ 
and we will compute the derivative 
of the Abel-Jacobi map ${\rm AJ}$ on $C$. 
First note that if we take a 
Weierstrass point $q$ of $C$ 
we have 
\[
{\rm AJ}(r(t)-j(r(t)))=2{\rm AJ}(r(t)-q).
\]
Take a coordinate
$z$ centered at $p$  
such that $\omega(z)=h(z)dz$ with $h(0)\neq 0.$ 
Fix a loop $r(t)$ such that $z(r(t))=t$, then
\[
\lim_{t\to0}\frac{1}{t} \int_{q}^{r(t)}\omega= 
\lim_{t\to 0}\frac{1}{t}\int_0^th(z)dz=h(0)
\] 
and we complete our result.
\end{proof}

\begin{corollary}
\label{cor}
The locus of curves $C$ in $\mathcal M_4$ 
such that $\eta_C$ is a non-trivial torsion point is a 
countable union of subvarieties 
of complex dimension $\geq 5$ 
and the set of subvarieties of
dimension $5$ is dense in $\mathcal M_4$
in the analytic topology.
\end{corollary}

\begin{proof}
Let $\mathcal U$ be an open subset of $\mathcal M_4$
which intersects the hyperelliptic locus and let $\tilde{\mathcal U}$ 
be the moduli space of pairs $(C,g_3^1)$,
where $[C]\in \mathcal U$. 
Let $\imath\colon\mathcal H\to\tilde{\mathcal U}$ be the 
subvariety containing pairs where $C$ is hyperelliptic. 
Given a universal family 
$\pi: \mathcal C\to  \tilde{\mathcal U}$,
we construct the following commutative diagram
 \[
 \xymatrix{
  \imath^*\mathcal C\ar[rrr]^-{(C,p)\mapsto [j(p)-p]}\ar[d]
  &&& \mathcal J\ar@{=}[d]\ar[rd]\\
  \mathcal C\ar[d]^-\pi\ar[rrr]^-{(C,g_3^1)\mapsto \eta_C} 
  &&& \mathcal J\ar[r]^-{\varphi} 
  &  \tilde{\mathcal U}\times
      \mathbb T \ar[d]^-{{\rm pr}_2}\\
 \tilde{\mathcal U}
  \ar[rrrr]^-{[C]\mapsto{\rm pr}_2(\varphi(\eta_C))}
  &&&& \mathbb T
 }
\]
where $\eta_C={g'}_3^1-g_3^1=K_C-2g_3^1$ and $\varphi$ is a $\mathcal C^\infty$-trivialization
(defined after possibly shrinking $\mathcal U$).
The commutativity of the top square comes from
the fact that on a hyperelliptic curve $C$ we have
$g_3^1 = p + g_2^1$ and $K_C - 2g_3^1 = j(p)-p$,
where $j\in{\rm Aut}(C)$ is the hyperelliptic involution.
By Theorem~\ref{diff} the differential 
of  the map $\gamma\colon\imath^*\mathcal C\to \mathbb T$ 
obtained composing the maps in the diagram 
is surjective at any point $p$ which is not Weierstrass. 
This implies that the map $\eta: \tilde{\mathcal U}\to \mathbb T$ 
is locally a submersion at any point corresponding to a hyperelliptic 
curve, in particular its image contains an open subset of $\mathbb T$.
We thus conclude as in the last part of the 
proof of Theorem~\ref{teo-2} given in section~\ref{density}.
\end{proof}


\section{Examples}
\label{exa}

In this section we will provide further examples 
of curves having two $g_k^1$'s whose 
difference is a torsion element in the Jacobian.
In particular we will show how to use automorphism
groups to construct new examples (see
Example~\ref{exa:new}).

\begin{proposition}\label{auto}
Let $C$ be a curve in $\mathcal F_k$.
If $G$ is an automorphism group  of $C$
of order $n$ which preserves each 
$g_k^1$ of $C$ and such that $C/G$ 
has genus zero, then the order of 
$\eta_C$ divides $n$.
\end{proposition}
\begin{proof}
Let $\pi:C\to C/G\cong \pp^1$ be 
the quotient morphism, let 
$D=p_1+p_2+\dots+p_k$ be an 
element of the first $g_k^1$
and let $q_i=\pi(p_i)$. Then
the following linear equivalences
hold
\[
 nD
 \sim
 \sum_{\sigma\in G} \sigma^*(D)
 =
 \pi^*(q_1)+\pi^*(q_2)+\dots+\pi^*(q_k)
 \sim
 kF, 
\]
where $F$ is a fiber of $\pi$ and the first equivalence 
is due to the fact that $G$ preserves the 
linear series $g_k^1$.
Since the same property holds for an element
$D'$ of the second $g_k^1$, the linear equivalence
$nD\sim nD'$ follows.
\end{proof}

\begin{example}
Let $\sigma$ be the order $k$ automorphism
of $\pp^1\times\pp^1$ defined by
\[
 \sigma(x_0,x_1,y_0,y_1)=(\zeta_k x_0,x_1, y_0,y_1),
\]
where $\zeta_k$ is a primitive $k$-th root of unity.
We now show that a curve $C\in\mathcal F_k$ 
which is $\sigma$-invariant
admits an equation of the form
\[
 x_0^kg_2(y_0,y_1)+x_1^kf_2(y_0,y_1)=0,
\]
where $f_2, g_2$ are homogeneous of degree 
$k$ in $y_0,y_1$. In particular the quotient
$C/\langle\sigma\rangle$ 
has genus zero. Thus  $C\in\mathcal F_k^{\rm tor}$
by either Proposition~\ref{auto} or Proposition~\ref{torsion}.
The automorphism $\sigma$ preserves 
each ruling of the quadric and acts identically 
on one of the two rulings.
Consider a point in $\mathcal H_1\cap\mathcal H_2$,
with the notation in the proof of Proposition~\ref{torsion}, 
which corresponds to a $\sigma$-invariant grid.
The lines of the grid which belong to the first ruling  
are defined by either 
$x_0^k-x_1^k=0$ or $x_0^k=0$.
An equation of $C$ in such coordinates 
is then of the form
\[
 (x_0^k-\mu x_1^k)h_1+(x_0^k+\lambda x_1^k)h_2
 =
 x_0^k(h_1+h_2)+x_1^k(\lambda h_2-\mu h_1)
 =
 0,
\]
for $\mu\in \{0,1\}$, $\lambda\in \cc$ and $h_1,h_2$ 
homogeneous of degree three in $y_0,y_1$.
\end{example}

\begin{example}
\label{exa:new}
The moduli space of non-hyperelliptic 
curves $C$ of genus four having an order five 
automorphism $\sigma$ such that 
$C/\langle\sigma\rangle$ 
has genus zero is a 1-dimensional subvariety 
of $\mathcal F_3^{\rm tor}$.
Moreover, any such $C$ is isomorphic to a 
curve in the following family 
\[
 x_0x_1^2y_1^3
 +\alpha x_0^2x_1y_0^3
 +\beta x_0^3y_0y_1^2
 +\gamma x_1^3y_0^2y_1=0,
\]
where $\sigma(x_0,x_1,y_0,y_1)
=(\zeta_5x_0,x_1,\zeta_5^3y_0,y_1)$ 
and $\zeta_5$ is a primitive fifth root of unity.
The family contains curves which pass through the points 
of a grid of type $(5,5)$, for example the curve with
\[
\alpha=-\zeta_5,\ \beta=\zeta_5^3+\zeta_5^2+\zeta_5,\ \gamma=\zeta_5^2+\zeta_5. 
\]
However, the general element of the family is not 
of grilled type.
This means that if $D_i$ is a divisor of the $i$-th $g_3^1$ 
and $\mathcal H_i\subseteq |5D_1|\cong\pp^{11}$ 
is the projectivization of the fifth symmetric 
power of $H^0(C,D_i)$, then the intersection
$\mathcal H_1\cap\mathcal H_2$ is empty.
For example this holds for the curve with $\alpha=-1, \beta=\gamma=1$.
The statements for both curves can be checked 
by means of the Magma~\cite{Magma} program
available here~\url{http://www2.udec.cl/~alaface/software/semiample/aut}.
\end{example}


\end{document}